\documentclass{article}
\usepackage[utf8]{inputenc}
\usepackage{amsmath}
\usepackage{amsthm}
\usepackage{amssymb}
\usepackage{tikz}
\usetikzlibrary{arrows, automata}
\usepackage{hyperref}
\usepackage{authblk}

\newtheorem{conjecture}{Conjecture}
\newtheorem{definition}{Definition}

\newtheorem{obs}{Observation}
\newtheorem{theorem}{Theorem}
\newtheorem{lemma}{Lemma}
\newtheorem{prop}{Proposition}
\newtheorem{corl}{Corollary}
\providecommand{\subjclass}[1]{\textbf{Mathematics Subject Classification:} #1}
\providecommand{\keywords}[1]{\textbf{Keywords:} #1}

\title{Vertices with the Second Neighborhood Property in Eulerian Digraphs}
\author{Michael Cary}
\begin{document}
\maketitle
\begin{abstract}
The Second Neighborhood Conjecture states that every simple digraph has a vertex whose second out-neighborhood is at least as large as its first out-neighborhood, i.e. a vertex with the Second Neighborhood Property. A cycle intersection graph of an even graph is a new graph whose vertices are the cycles in a cycle decomposition of the original graph and whose edges represent vertex intersections of the cycles. By using a digraph variant of this concept, we prove that Eulerian digraphs which admit a simple cycle intersection graph have not only adhere to the Second Neighborhood Conjecture, but that local simplicity can, in some cases, also imply the existence of a Seymour vertex in the original digraph.
\end{abstract}



\section{Introduction}
The Second Neighborhood Conjecture states that every simple digraph $D$ has at least one vertex with the Second Neighborhood Property, i.e. a vertex $v$ satisfying $|N^{+2}(v)|\geq |N^{+}(v)|$ in $D$. Such vertices are often called Seymour vertices in the literature \cite{cohn}. For notational purposes, $SV(D)$ will be used to denote the set of vertices having the second neighborhood property in a given digraph.

Progress towards answering this conjecture began with a proof of Dean's Conjecture, which states that the Second Neighborhood Conjecture for tournaments \cite{dean,fisher}. A new proof of Dean's Conjecture using local median orders was given later in \cite{havet}. Ever since, further results have primarily consisted of applications of local median orders to tournaments missing a well-defined structure such as a generalized star or a subset of the arc set of a smaller tournament \cite{ghazal,dara2018second}, or have dealt strictly with dense digraphs \cite{fidler}. Markov graphs are known to satisfy the Second Neighborhood Conjecture \cite{kozerenko2018expansive}. Progress has also been made on quasi-transitive graphs\cite{li2018second}. Additionally, it has been shown in \cite{chen} that every simple digraph has a vertex whose second out-neighborhood is at least $\gamma=0.657298\dots$ where $\gamma$ is the unique real root of the equation $2x^{3}+x^{2}-1=0$. A recent equivalent formulation of the conjecture may be found in \cite{seacrest2018seymour}. For a more complete summary of results pertaining to the Second Neighborhood Conjecture, the interested reader is referred to \cite{sullivan}.

The approach this paper takes is to use cycle decompositions of Eulerian digraphs admitting simple dicycle intersection graphs. Cycle intersection graphs were introduced in \cite{cary2018cycle}. The cycle intersection of an even graph $G$ is a graph $CI(G)$ whose vertex set is the set of cycles in a particular cycle decomposition of $G$ and whose edges represent unique intersections (vertices) of cycles in the cycle decomposition of $G$. Clearly this concept can be extended to cycle decompositions for Eulerian digraphs. An example of an even graph $G$ and a cycle intersection graph $CI(G)$ of $G$ is presented below. By orienting the cycles we obtain a dicycle intersection graph $CI(D)$ for an orientation $D$ of $G$. For more on decompositions of digraphs and of graphs in general, one may see \cite{bang,zhang}. 

\begin{figure}[h!]
\centering
\begin{tikzpicture}[-,>=stealth',shorten >=1pt,auto,node distance=2cm,
                    thick,main node/.style={circle,draw}]
  \node[main node] (A)                      {};
  \node[main node] (B) [below right of=A]   {};
  \node[main node] (C) [below of=B]         {};
  \node[main node] (D) [below left of=A]    {};
  \node[main node] (E) [below of=D]         {};
  \draw[thick,-] (A) to (B);
  \draw[thick,-] (B) to (C);
  \draw[thick,-] (C) to (A);
  \draw[thick,-] (A) to (D);
  \draw[thick,-] (D) to (E);
  \draw[thick,-] (E) to (A);
\end{tikzpicture}
\caption{An example of an even graph $G$ that does admit a single cycle decomposition whose associated cycle intersection graph is simple. The only possible cycle intersection graph of $G$ is $K_{2}$. If, instead, we considered $G=K_{5}$ the graph $5K_{2}$ is a cycle intersection graph of $G$, though it is not unique.}
\label{fig1}
\end{figure}
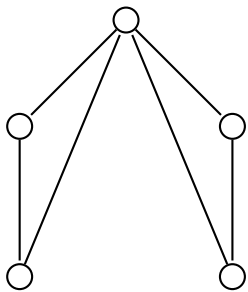

Clearly Eulerian digraphs have an even underlying graph, but they also admit cycle decompositions. In this paper we consider those Eulerian digraphs which admit a simple dicycle intersection graph. The reason for assuming that the dicycle intersection graph of an Eulerian digraph is simple is that in doing so we guarantee that any two dicycles in a cycle decomposition intersect in at most one vertex. This restriction does not allow us to prove that all Eulerian digraphs necessarily have a Seymour vertex, but it may be possible that this method can be extended.

One related extension using this technique could be to prove a variant of the Second Neighborhood Conjecture that applies specifically to Eulerian digraphs. The conjecture, which appears as Conjecture 6.15 in \cite{sullivan}, asserts that the average size of the second out-neighborhood of vertices in an Eulerian digraph is at least as large as the average size of the first out-neighborhood of its vertices.

\begin{conjecture}\label{conjecture}
If $D$ is an Eulerian digraph with no loops or digons, then
\begin{equation*}
\sum\limits_{v\in V(D)}|N^{+2}(v)|\geq \sum\limits_{v\in V(D)}|N^{+}(v)|
\end{equation*}
\end{conjecture}

By studying Eulerian digraphs we allow for a different approach than what has been taken by past researchers. Broadly speaking, previous works have started with a tournament and removed convenient substructures while showing that the smaller graph has a Seymour vertex. Our approach allows us to use a relatively large substructure (an Eulerian digraph) which has a Seymour vertex as a starting point to build from. Since $r$-regular digraphs are necessarily Eulerian, our results are a large step towards showing that $r$-regular digraphs have at least one vertex with the second neighborhood property.

\section{Seymour vertices in Eulerian digraphs}
Our first main result will be to prove that all Eulerian digraphs which admit a simple dicycle intersection graph have at least one Seymour vertex, so we begin by presenting relevant definitions. Additionally, all digraphs in this paper are assumed to be simple and connected.

\begin{definition}\label{d1}
An even directed graph is any directed graph $D$ with the property that $d^{+}(v)=d^{-}(v)$ for all $v\in V(D)$.
\end{definition}

\begin{definition}\label{d2}
An Eulerian digraph is any digraph having a directed Eulerian tour.
\end{definition}

Our first proposition states that these two families of digraphs are equivalent. This result may be found in \cite{bollobas}.

\begin{prop}\label{p1}
\cite{bollobas} A connected digraph is Eulerian if and only if $d^{+}(v)=d^{-}(v)$ for all $v\in V(D)$.
\end{prop}

The following observation is mentioned as it is as necessary as it is obvious for the proof of our main result.

\begin{obs}\label{o1}
Let $D$ be an even digraph and let $C$ be a dicycle of $D$. Then the digraph $D\setminus A(C)$ is an even digraph.
\end{obs}

The main tool used in the proof of our first theorem is graph decompositions. Throughout this paper we will denote a cycle decomposition of a digraph $D$ by $\mathcal{F}(D)$. In particular, we will need to use the fact that Eulerian digraphs decompose into arc-disjoint dicycles. A proof of this fact is presented in the following lemma.

\begin{lemma}\label{l1}
If $D$ is an Eulerian digraph, then $D$ decomposes into arc-disjoint dicycles, i.e. $D$ has a cycle decomposition $\mathcal{F}(D)$.
\end{lemma}
\begin{proof}
We may consider only connected digraphs, since if this holds for any connected component of a digraph, it will hold for the entire digraph. The proof is by induction on $\Delta^{+}(D)$. When $\Delta^{+}(D)=1$, then $D$ is simply a dicycle and thus trivially has a cycle decomposition $\mathcal{F}(D)=\{D\}$. Assume that for all digraphs having $d^{+}(v)=d^{-}(v)$ for all $v\in V(D)$ and maximum out-degree less than $n$, $D$ has a cycle decomposition $\mathcal{F}(D)$. Now let $D$ be a digraph such that $d^{+}(v)=d^{-}(v)$ for all $v\in V(D)$ and $\Delta^{+}(D)=n$. Remove a dicycle passing through each vertex of $D$ having maximum out-degree in $D$ and call this set of dicycles $S$. The digraph $D^{\prime}=D\setminus S$ has $\Delta^{+}(D^{\prime})=\Delta^{+}(D)-1$ and so, by our inductive hypothesis, has a cycle decomposition, say $\mathcal{F}(D^{\prime})$. Then $\mathcal{F}(D)=\mathcal{F}(D^{\prime})\cup S$ is a cycle decomposition of $D$ and the proof is complete.
\end{proof}

The last tool we will need in order to prove our main result is an old result proven by Kaneko and Locke in \cite{kaneko}. For concision, the interested reader is referred to this paper for a proof of the result.

\begin{lemma}\label{l2}
\cite{kaneko} Let $D$ be a simple digraph with no loops or digons. If $\delta^{+}(D)\leq 6$, then $SV(D)\neq\emptyset$.
\end{lemma}

We are now ready to prove our first main result. But rather than proving that some vertex in an Eulerian digraph with a simple dicycle intersection graph satisfies the Second Neighborhood Conjecture, we actually prove that \emph{every} vertex in such an Eulerian digraph has the second neighborhood property.

\begin{theorem}\label{t1}
If $D$ is an Eulerian digraph which admits a simple dicycle intersection graph $CI(D)$ then $SV(D)=V(D)$.
\end{theorem}
\begin{proof}
Let $D$ be an Eulerian digraph which admits a simple dicycle intersection graph and consider any vertex $v$. Let $vxy\subseteq C$ for some dicycle $C$ in a cycle decomposition $\mathcal{F}(D)$ of $D$ which corresponds to a simple dicycle intersection graph. If some other dicycle in $\mathcal{F}(D)$ contains both $v$ and $y$ then the dicycle intersection graph corresponding to our choice of $\mathcal{F}(D)$ has a digon, a contradiction. Therefore we may conclude that $SV(D)=V(D)$ for Eulerian digraphs which admit a simple dicycle intersection graph.  
\end{proof}

An immediate corollary of this is that Conjecture \ref{conjecture} is true for this subclass of Eulerian digraphs.

\begin{corl}
Conjecture \ref{conjecture} is true for Eulerian digraphs which admit at least one simple dicycle intersection graph.
\end{corl}

It is not the case that a simple dicycle intersection graph is necessary for $SV(D)=V(D)$. As the following theorem will show, an example of this would be the $2$-regular tournament $T_{5}$.

\begin{theorem}
If $T_{n}$ is an $r$-regular tournament then $SV(D)=V(D)$.
\end{theorem}
\begin{proof}
Assume not. Let $v$ not be a Seymour vertex of $T_{n}$. Since $d^{+}(v)=r=\frac{n-1}{2}$, there must exist some vertex $x$ such that $x\not\in N^{+}(v)\cup N^{+2}(v)$. But then $x$ must dominate both $v$ and $N^{+}(v)$, contradicting that $T_{n}$ is $r$-regular.
\end{proof}

Extending this result on regular tournaments, we also state that nearly regular tournaments, i.e., tournaments on $n\equiv 0\ (\mathrm{mod}\ 2)$ vertices whose vertices all have in-degree and out-degree in the set $\{\frac{n}{2}-1,\frac{n}{2}\}$, satisfy $SV(T_{n})=\{v\in V(T_{n})|d^{+}(v)=\delta^{+}(T_{n})\}$.

\begin{theorem}
If $T_{n}$ is a nearly regular tournament then $SV(T_{n})=S$ where $S=\{v\in V(D)|d^{+}(v)=\delta^{+}(D)\}$.
\end{theorem}
\begin{proof}
Assume not. Let $S=\{v\in V(T_{n})|d^{+}(v)=\delta^{+}(T_{n})\}$. Clearly no vertex not in $S$ can be a Seymour vertex, hence we may assume that there exists some vertex $v\in S$ such that $|N^{+2}|\leq(\frac{n}{2}-2)$. But then there exist vertices $x$ and $y$ such that neither $x$ nor $y$ is a member of $N^{+}(v)\cup N^{+2}(v)$. This implies that each of $x$ and $y$ dominate both $v$ and $N^{+}(v)$. Since an arc exists between $x$ and $y$, at least one of $x$ or $y$ has out-degree at least $|N^{+}(v)|+2=\frac{n}{2}+1$ which is impossible, Therefore we conclude that $SV(T_{n})=S$.
\end{proof}

To build on these examples of digraphs whose cycle intersection graphs are not simple yet have Seymour vertices, we turn our attention to substructures of the cycle intersection graph which imply the existence of a Seymour vertex in the original digraph. To do this, we require the use of closed out-neighborhoods. The closed out-neighborhood of a vertex $x\in V(D)$ is simply the union of the out-neighborhood $N^{+}(v)$ of $v$ with the vertex $v$ itself. Notationally, we say that $N^{+}[v] = N^{+}(v)\cup\{v\}$.

Intuitively, we want to look for substructures of cycle intersection graphs which are simple. In particular, we want induced subgraphs that are simple because such substructures minimize the potential for two out-neighbors of a vertex to share a common out-neighbor. This is an issue because this contributes two members to $N^{+}(v)$ but only one member to $N^{+2}(v)$. An illustration of this is presented below.

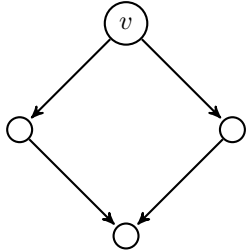
\begin{figure}[h!]
\centering
\begin{tikzpicture}[-,>=stealth',shorten >=1pt,auto,node distance=2cm,
                thick,main node/.style={circle,draw}]
\node[main node] (A)                      {$v$};
\node[main node] (B) [below right of=A]   {};
\node[main node] (C) [below left of=A]    {};
\node[main node] (D) [below left of=B]    {};
\draw[thick,->] (A) to (B);
\draw[thick,->] (A) to (C);
\draw[thick,->] (C) to (D);
\draw[thick,->] (B) to (D);
\end{tikzpicture}
\caption{An example of a vertex $v$ for which two members of $N^{+}(v)$ lead to a common member of $N^{+2}(v)$. This phenomenon is the crux of the Second Neighborhood Conjecture.}
\label{diamond}
\end{figure}

A natural substructure to look for would be an induced simple clique. We emphasize simple here because multiple edges can be quite common in cycle intersection graphs and we do not permit any multiple edges in this analysis. Induced simple cliques may appear in cycle intersection graphs for several reasons. One possibility is that a set of dicycles in a cycle decomposition all intersect pairwise at a single unique vertex, while another possibility is that a series of otherwise disjoint dicycles all intersect at a single vertex. For this next result we only consider the latter case.

\begin{theorem}
If an Eulerian digraph $D$ has a cycle intersection graph $CI(D)$ which admits an induced simple clique $CI(D)[N[\hat{v}]]$ which corresponds to a single vertex $v\in V(D)$ then $v\in SV(D)$.
\end{theorem}
\begin{proof}
Let $\mathcal{F}(D)$ be the cycle decomposition associated with $CI(D)$. Since the induced cycle $CI(D)[N[\hat{v}]]$ is simple, none of the dicycles of $D$ which correspond to the vertices of $CI(D)[N[\hat{v}]]$ intersect at any other vertex of $D$ besides $v$. It then follows that $v$ has a unique first and second out-neighbor in each of these dicycles. Additionally, $v$ is not contained in any other dicylces of $\mathcal{F}(D)$ beyond those represented by the vertices comprising the induced simple clique $CI(D)[N[\hat{v}]]$, for otherwise the vertices representing these additional cycles would be included in the induced subgraph $CI(D)[N[\hat{v}]]$. Since this comprises the entire first out-neighborhood of $v$ in $D$, it follows that $v\in SV(D)$. 
\end{proof}

This result gives us a well defined member of $SV(D)$. By expanding the scope of the induced simple subgraph from a clique to a block graph (a graph in which every biconnected component is a clique), we can obtain a stronger result on $SV(D)$.

\begin{theorem}
If an Eulerian digraph $D$ has a cycle intersection graph $CI(D)$ which admits an induced simple block graph $CI(D)[N[\hat{v}]]$ then $SV(D)\neq\emptyset$. In particular, $V(C_{\hat{v}})\subseteq SV(D)$ where $C_{\hat{v}}$ is the dicycle of $D$ corresponding to the vertex $\hat{v}\in V(CI(D))$.
\end{theorem}
\begin{proof}
If $CI(D)[N[\hat{v}]]$ is a simple block graph, then every dicycle which intersects $C_{\hat{v}}$ does so at precisely one vertex and does not intersect any other dicycle which intersects $C_{\hat{v}}$ at any vertex not in $V(C_{\hat{v}})$, for otherwise $CI(D)[N[\hat{v}]]$ would not be simple. Thus each vertex of $C_{\hat{v}}$ has a unique first and second out-neighbor in each cycle which passes through it, confirming that $V(C_{\hat{v}})\subseteq SV(D)$.
\end{proof}

These last two results are based on the local non-existence of multiple edges in the cycle intersection graph. A natural extension of these results would be to study local structures with relatively few multiple edges.

\section{Eulerian Subdigraphs}
In this section we turn our attention to discussing Eulerian subdigraphs of larger directed graphs. We introduce and use terminology pertinent to our previous theorem in order to characterize digraphs in terms of certain Eulerian subdigraphs.

\begin{definition}\label{d3}
A skeleton is a maximal Eulerian subdigraph of a directed graph.
\end{definition}
\begin{definition}\label{d4}
A vertebrate is any digraph with a non-empty skeleton.
\end{definition}
\begin{definition}\label{d5}
An invertebrate is any digraph whose skeleton is the empty set.
\end{definition}

\begin{prop}\label{p2}
$D$ is an invertebrate if and only if $D$ is a directed acyclic graph (DAG).
\end{prop}
\begin{proof}
$(\implies)$ Since $D$ is an invertebrate, a maximum skeleton of $D$ is the empty set, hence $D$ is acyclic.

$(\impliedby)$ Since $D$ is a DAG, $D$ contains no dicycles. Since $D$ contains no dicycles, a maximum skeleton must be the empty set, hence $D$ is an invertebrate.
\end{proof}

From this follows the obvious corollary that invertebrates always have at least one Seymour vertex since it is well known that DAGs always have at least one Seymour vertex.

\begin{corl}\label{corl2}
If $D$ is an invertebrate, then $SV(D)\neq\emptyset$.
\end{corl}

Another result immediately following from these new definitions is as follows.
\begin{theorem}\label{t22}
Let $D$ be a digraph which admits a skeleton $S(D)$ such that $S(D)$ admits a simple dicycle intersection graph. If there exists some vertex $v$ such that $d_{D}^{+}(v)=d_{S(D)}^{+}(v)$ then $v\in SV(D)$.
\end{theorem}
\begin{proof}
By Theorem \ref{t1} we know that every vertex in $S(D)$ is a Seymour vertex in $S(D)$. Any vertex satisfying $d_{D}^{+}(v)=d_{S(D)}^{+}(v)$ incurs no new out-neighbors from $S(D)$ to $D$ and thus retains the second neighborhood property in $D$.
\end{proof}
\begin{corl}
If a digraph $D$ admits a skeleton $S(D)$ which has a simple dicycle intersection graph and at least one of $\delta^{+}(D)=\delta^{+}(S(D))$ or $\Delta^{+}(D)=\Delta^{+}(S(D))$ is true, then $SV(D)\neq\emptyset$.
\end{corl}

The above corollary is a special case of the previous theorem which presents an easy method to find a vertex with the second neighborhood property.

\section{Conclusion}
In this paper we showed that the Second Neighborhood Conjecture is true for Eulerian digraphs which admit a simple cycle intersection graph.  We also proved that the existence of certain induced subgraphs of the cycle intersection graph imply the existence of a Seymour vertex in the original digraph. Using the notion of Eulerian subdigraphs of digraphs, we show if a digraph $D$ which contains an Eulerian subdigraph $S(D)$ which admits a simple cycle intersection graph $CI(S(D))$ and $\delta^{+}(D)=\delta^{+}(S(D))$, then the digraph $D$ itself has a Seymour vertex. A natural extension of these results would be to determine if digraphs that do not admit any dicycle intersection graphs which are simple still have a Seymour vertex, in particular one which satisfies $d^{+}(v)=\delta^{+}(D)$. To do this, studies on the structural properties of cycle intersection graphs would be a useful direction for future research, in particular the role that certain induced subgraph of cycle intersection graphs have in determining the existence of a Seymour vertex in the original digraph. Additionally, the approach of using skeletons as a basis and adding structures to them to prove that digraphs admit Seymour vertices appears to be a promising approach to the Second Neighborhood Conjecture.

\begin{flushleft}
Michael Cary\\
macary@mix.wvu.edu\\
\quad\\
West Virgnia University,\\
Morgantown, WV, USA
\end{flushleft}
\end{document}